\RequirePackage{fix-cm}
\documentclass[smallextended]{svjour3}       
\smartqed  
\usepackage{graphicx}
\usepackage{amssymb}
\usepackage{amsmath}
\usepackage{subcaption,multicol}

\newcommand{\bN}{\mathbb{N}}
\newcommand{\bL}{\mathbb{L}}
\newcommand{\bZ}{\mathbb{Z}}
\newcommand{\mP}{\mathcal{P}}
\newcommand{\mPlk}{\mathcal{P}_{l,k}(C_n(a,b))}
\newcommand{\bSnab}{\bar{S}_n(a,b)}
\newcommand{\Snab}{S_n(a,b)}
\newcommand{\st}{\text{st}}
\newcommand{\primwords}{\mathbb{P}_2(l,k)}
\newcommand{\allwords}{\mathbb{W}_2(l,k)}
\newcommand{\nonpwords}{\mathbb{P}^c_2(l,k)}
\newcommand{\lyndonwords}{\mathbb{L}_2(l,k)}
\newcommand{\lyn}{\text{Lyn}}
\newcommand{\ptwo}{\mathbb{P}^c_2}
\begin{document}


\title{Periodic orbits on 2-regular circulant digraphs
}


\author{Isaac Echols \and Jon Harrison \and Tori Hudgins}


\institute{I. Echols \and J. Harrison (corresponding author) \at
              Department of Mathematics, Baylor University, 1410 S. 4th Street, Waco, 76706, TX, USA \\
              \email{isaac\_echols1@baylor.edu, jon\_harrison@baylor.edu}           
           \and
           T. Hudgins \at Department of Mathematics, University of Kansas, 405 Snow Hall, 1460 Jayhawk Blvd., Lawrence, 66045, KS, USA\\
           \email{thudgins@ku.edu}
}

\date{Received: date / Accepted: date}

\maketitle

\begin{abstract}
Periodic orbits (equivalence classes of closed paths up to cyclic shifts) play an important role in applications of graph theory.  For example, they appear in the definition of the Ihara zeta function and exact trace formulae for the spectra of quantum graphs. Circulant graphs are Cayley graphs of $\mathbb{Z}_n$.  Here we consider directed Cayley graphs with two generators ($2$-regular Cayley digraphs). We determine the number of primitive periodic orbits of a given length (total number of directed edges) in terms of the number of times edges corresponding to each generator appear in the periodic orbit (the step count). Primitive periodic orbits are those periodic orbits that cannot be written as a repetition of a shorter orbit. We describe the lattice structure of lengths and step counts for which periodic orbits exist and characterize the repetition number of a periodic orbit by its winding number (the sum of the step sequence divided by the number of vertices) and the repetition number of its step sequence. To obtain these results, we also evaluate the number of Lyndon words on an alphabet of two letters with a given length and letter count.
\keywords{Graph Theory \and Circulant Graphs \and Lyndon Words \and Periodic Orbits}
\subclass{05C38 \and 68R15 \and 81Q50}
\end{abstract}

\section{Introduction}
\label{intro}


Periodic orbits play an important role in graph theory, where the Ihara zeta function is defined via an Euler product over primitive backtrack-less tail-less cycles, see e.g. \cite{S86,H89,ST96}.  Graphs are also widely employed in mathematical physics, for example to model nanomaterials, waveguides, photonic crystals and quantum chaos, see, e.g. \cite{K04,berkolaiko2013introduction,specgeomgraphs}.  One advantage of quantum graphs is that they are endowed with an exact trace formula, where a spectral quantity is expressed via a sum over periodic orbits (cycles) on the graph.  Such a trace formula was first derived by Roth \cite{R83} for the heat kernel and independently by Kottos and Smilansky for the density of states \cite{KS97,KS99}, with subsequent extensions \cite{BH03,KPS07,K08,BE09}.

In addition to the exact trace formula, there is an exact expansion of the characteristic polynomial of a graph as a finite sum over pseudo orbits (sets of periodic orbits) on the graph.  This was originally described by Kottos and Smilansky \cite{KS99} and refined with an explicit formulation in terms of irreducible pseudo orbits (pseudo orbits where each bond appears at most once) \cite{BHJ12}.  Similarly, Akkermans et al. \cite{ACDMT00} derive the spectral determinant of the graph as a product over periodic orbits.
The pseudo orbit expansion can be employed to evaluate the variance of coefficients of the characteristic polynomial \cite{BAND2019135}.  In particular, for a quantum graph with Fast Fourier Transform scattering matrices at the vertices the variance can be evaluated precisely in terms of the sizes of certain sets of pseudo orbits with self-intersections \cite{HH22b,HH22}.  As a first step to extend such exact combinatorial calculations of spectral statistics, in this article, we find methods to categorize and count the primitive periodic orbits on $2$-regular circulant directed graphs.

Circulant graphs with $n$ vertices are Cayley graphs of the cyclic group $\mathbb{Z}_n$.
The distribution of the diameters of random circulant graphs was investigated by Marklof and Str\"ombergson \cite{MS13} and spectral properties of quantum circulant graphs were described in \cite{HS19}.

Here we consider $2$-regular directed circulant graphs and count the number of primitive periodic orbits (periodic orbits that are not a repetition of a shorter orbit). 
A $2$-regular circulant graph $C_n(a,b)$ on $n$ vertices is defined by a choice of two step sizes $0<a<b<n$.
An edge runs from an origin vertex $i\in \mathbb{Z}_n$ to a terminal vertex $j\in \mathbb{Z}_n$ if $j=i+a$ or $j=i+b$ (where addition is modulo $n$), so edges fall into two class corresponding to the two step sizes $a$ and $b$.   Figure \ref{fig:cgexamples} shows examples of directed circulant graphs. The length $l$ of a path on the graph is the number of edges in the path and we use $k$ to denote the number of times the path uses edges from class $b$ (the $b$-count).  Hence a path is closed if $n|(l-k)a+kb$.  
In this case $la + k(b-a) = \omega n$ for an integer $\omega$ and we call $\omega$ the winding number of the path.  We can now state the main results. 

\begin{figure}[htb]
  \includegraphics[width = \linewidth]{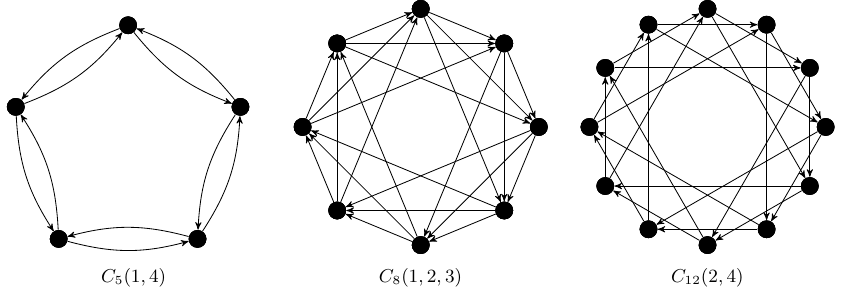}
\caption{Examples of directed circulant graphs}
\label{fig:cgexamples}
\end{figure}

\begin{theorem}\label{thm: old main} 
Let $C_n(a,b)$ be a connected $2$-regular circulant digraph, and suppose $la + k(b-a) = \omega n$. The number of primitive periodic orbits of length $l$ with $b$-count $k$ is, 
    \begin{equation*}
        \frac{n}{l}\sum_{m|\gcd(l,k,\omega)}\mu(m)\binom{l/m}{k/m} \ ,
    \end{equation*}
    where $\mu$ is the M\"obius function. 
\end{theorem}
Then summing over the allowed $b$-counts we obtain the total number of primitive periodic orbits of length $l$, a formula proposed with numerical evidence in \cite{EHH}.
\begin{theorem}\label{thm: new main}
Let $C_n(a,b)$ be a connected $2$-regular circulant digraph.  Then the number of primitive periodic orbits of length $l$ is 
    $$ \frac{n}{l}\sum_{\omega = \lceil la/n \rceil}^{\lfloor lb/n \rfloor} \sum_{m|\omega}\mu(m)\binom{l/m}{k_\omega/m} \ , $$
    where $k_\omega = (\omega n-la)/(b-a)$, and $\mu$ is the M{\"o}bius function.  We interpret $\binom{x}{y}$ as zero if $x,y\notin\bN_0$.
\end{theorem}

While theorems \ref{thm: new main} and \ref{thm: old main} count the total number of primitive periodic orbits of a given length, and the number with a given length and $b$-count respectively, they do not immediately simplify the trace formula, or formulas for the spectral statistics.   These formulas are not typically entirely combinatorial.  For example, they contain phases that are dependent on the metric length of the periodic orbits.  If one considers an equilateral graph where all the edges have the same metric length this phase would depend only on the number of edges in the orbit.  However, the trace formula still contains a stability amplitude for each orbit, the product of elements of the vertex scattering matrices around the orbit.  The stability amplitude can vary between periodic orbits, even for orbits with the same length and $b$-count.  Despite this, in certain situations, it is possible to reduce spectral quantities to a purely combinatorial formula; for example, the variance of the coefficients of the characteristic polynomial of a $2$-regular directed quantum graph with fast Fourier transform vertex scattering matrices \cite{HH22}.  Consequently, we regard counting primitive periodic orbits as a first step toward simplifying the trace formula or evaluating spectral statistics of quantum circulant graphs.

To count primitive periodic orbits we exploit a connection with Lyndon words on an alphabet of two letters.  Lyndon words are words that are earlier (in lexicographic order) than all of their rotations.   In the process of counting primitive orbits we obtain the following formula for the number of Lyndon words on an alphabet of two letters with a given letter count. 
\begin{theorem}\label{thm: Lyndon}
Let $\mathbb{L}_2(l,k)$ be the set of Lyndon words of length $l$ on the alphabet $\{a,b\}$  where $b$ appears $k$ times.  Then,
\begin{equation}
    |\lyndonwords| = \frac{1}{l}\sum_{m|\gcd(l,k)}\mu(m)\binom{l/m}{k/m} \ .
\end{equation}
\end{theorem}

The article is organized as follows.  In section \ref{sec: circulant} we introduce the notation for circulant graphs and periodic orbits.  In section \ref{sec: lattice} we describe the lattice structure of the allowed values of periodic orbit lengths and $b$-counts corresponding to primitive periodic orbits.  In section \ref{sec: Lyndon} we obtain theorem \ref{thm: Lyndon} for the number of Lyndon words of length $l$ and $b$-count $k$.  In section \ref{sec: counting} we use theorem \ref{thm: Lyndon} to prove theorems \ref{thm: old main} and \ref{thm: new main} counting the number of primitive periodic orbits on $2$-regular circulant digraphs.  Finally, in section \ref{sec: conclusions}, we discuss the results and illustrate them with some examples.

\section{Circulant Graphs}\label{sec: circulant}

A \emph{directed graph} (digraph) is a set $\mathcal{V}$ of vertices, together with a set $\mathcal{B}$ of ordered pairs of vertices, called bonds (directed edges). For a bond $e=(v_1,v_2)$, the vertex $v_1$ is called the \textit{origin} of $e$, denoted $o(e)$, and the vertex $v_2$ is called the \textit{terminus} of $e$, denoted $t(e)$. In this paper, we assume graphs are simple; that is, $\mathcal{V}$ is a finite set and $\mathcal{B}$ has no repeated bonds or bonds of the form $(v,v)$ (loops). Note that we do allow pairs of bonds which are the reverse of each other, $(v,u)$ and $(u,v)$. 

Given integers $0<a_1<a_2<\dots<a_m < n$, the \textit{directed circulant graph} on $n$ vertices with step sizes $a_1,\dots,a_m$, denoted $C_n(a_1,\dots,a_m)$, is the directed graph with vertex set $\mathcal{V} = \bZ_n$  and bond set $\mathcal{B}=\{(v,v+a_i)|v\in \mathcal{V}, 1\leq i\leq m \}$, where addition in $\bZ_n$ is modulo $n$. Note that undirected circulant graphs require $a_m\leq n/2$ to prevent repeated edges.
Figure \ref{fig:cgexamples} shows the circulant digraphs $C_5(1,4)$, $C_8(1,2,3)$ and $C_{12}(2,4)$. For a bond $e = (v,v+a_i)$ where $0<a_i<n$, the \textit{step size} of $e$, denoted $|e|$, is the number $a_i$.
Note that $C_n(a_1,\dots,a_m)$ is \textit{$m$-regular}; that is, each vertex has $m$ incoming bonds and $m$ outgoing bonds.

On a directed graph, a \textit{path} from a vertex $v$ to a vertex $u$ is a sequence of bonds $c = (e_1,e_2,\dots,e_l)$ for which $o(e_1)=v$, $t(e_l)=u$, and $t(e_j) = o(e_{j+1})$ for each $j$. A path is called a \textit{circuit} if additionally $v = u$. The number of (possibly repeated) bonds $l$ is called the \textit{length} of the path. For a circulant graph, the \textit{transit distance}, denoted $\Delta(c)$, is $|e_1|+\dots+|e_l|$. Note that $u \equiv v + \Delta(c) \, (\text{mod } n)$. As such, $c$ is a circuit if and only if $n|\Delta(c)$, in which case we define the \textit{winding number} of $c$ to be $\Delta(c)/n$, denoted 
$\omega$. 
The \textit{step sequence} of a path $c$ is $\st(c)=(|e_1|,\dots,|e_l|)$. 

A directed graph is \textit{strongly connected} if, for any two vertices $v$ and $u$, there is a path from $v$ to $u$. In figure \ref{fig:cgexamples}, $C_{12}(2,4)$ is disconnected; it has two strongly connected components, namely $\{0,2,4,6,8,10\}$ and $\{1,3,5,7,9,11\}$.  In this paper, we consider circulant graphs that are strongly connected, or equivalently $\gcd(n,a_1,\dots,a_m) =1$, see \cite{D86}.

Given a circuit $c = (e_1,e_2,\dots,e_l)$, the \textit{periodic orbit} containing $c$, denoted $p = [c]$, is the equivalence class of $c$, up to rotation. That is, $p = \{\sigma^x(c): x = 0,\dots,l-1\}$, where $\sigma(e_1,e_2,\dots,e_l) = (e_2,e_3,\dots,e_l,e_1)$. Note that length, transit distance, and winding number of circuits in $p$ is invariant, so they are well-defined for $p$.
A periodic orbit $p$ is called \textit{primitive} if it is not a repetition of a shorter periodic orbit, that is, if there is no circuit $c_0$ and $r>1$ such that $p = [c_0^r]$. For nonprimitive $p$, there is a unique primitive periodic orbit $q$, called the \textit{root} of $p$, such that $p = q^r$, in which case $r$ is called the \textit{repetition number} of $p$. We write $\mathcal{P}_l(G)$ for the set of primitive periodic orbits of length $l$ in a directed graph $G$. 

In this article, we consider $2$-regular circulant graphs $C_n(a,b)$. We write $\mPlk$ for the set of primitive periodic orbits of length $l$ for which the number of bonds of step size $b$ is $k$; we refer to $k$ as the $b$-\textit{count} of an orbit or circuit. Note that these sets are disjoint for different values of $k$, so 
\begin{equation}\label{eq: sum over k basic}
    |\mP_l(C_n(a,b))| = \sum_{k=0}^l |\mPlk| \ .
\end{equation}
This sum can be simplified by classifying the values of $k$ for which the set $\mPlk$ is nonempty, which is the goal of the next section.

\section{The Lattice of Periodic Orbits}\label{sec: lattice}

In this section we classify periodic orbits in $C_n(a,b)$ by length $l$ and $b$-count $k$, and describe the lattice structure of all corresponding $(l,k)$ in $\bZ^2$.
If the set $\mPlk$ is nonempty then 
\begin{equation}\label{eq: la+kd}
    \Delta(p) = \omega_p n = (l-k)a+kb = la + kd \ ,
\end{equation} where $d = b-a$. Let $\bSnab = \{(l,k)\in \bZ^2:la+kd = \omega n \text{ for some } \omega \in \bZ\}$. This includes negative $l$ and $k$, so let $\Snab$ be the subset of $\bSnab$ that correspond to actual orbits, i.e., for which $0\leq k\leq l$ and $l>0$.

Rather than search for orbits directly in terms of length and $b$-count, it is more efficient to search by the winding number, which is a multiple of $g = \gcd(a,d) = \gcd(a,b)$.

\begin{proposition}
    \label{gdividesomega}
        Let $g = \gcd(a,d)$ and suppose $\gcd(a,d,n) = 1$. Then the equation $la + kd = \omega n$ has an integer solution $l,k$ if and only if $g|\omega$.
\end{proposition}
\begin{proof}
        By the Chinese remainder theorem, an integer solution $l,k$ exists if and only if $g|\omega n$. But $\gcd(g,n) = \gcd(a,d,n) = 1$. Then $g|\omega n$ if and only if $g|\omega$, as $g$ and $n$ are coprime. \qed
\end{proof}

Now that we know the set of possible winding numbers, we will find points in $\bSnab$ with a fixed winding number.

 \begin{proposition}
        Suppose $la+kd = t$ has an integer solution $l_0, k_0$. Then $l_1, k_1$ is another solution if and only if there exists $x \in \bZ$ such that $l_1 = l_0+xd'$ and $k_1 = k_0-xa'$, where $a' = a/g, d' = d/g$.
    \end{proposition}
    \begin{proof}
        Suppose $l_0a+k_0d = l_1a+k_1d$. Dividing by $g$, we have $l_0a' + k_0d' = l_1a'+k_1d'$. Then $x=(l_1-l_0)/d'=(k_0-k_1)/a'$. 
        To show $x\in \bZ$, note that $(l_1-l_0)a'/d' = k_0-k_1 \in \bZ$. Thus $d'|(l_1-l_0)a'$. Since $\gcd(a',d') = 1$, we have $d'|(l_1-l_0)$.

        Conversely, suppose $l_1 = l_0+xd'$, and $k_1 = k_0 - xa'$, with $x\in \bZ$. Then $l_1,k_1\in \bZ.$ Note that $d'a = a'd$. Thus $l_1a+k_1d = l_0a + x(d'a -a'd) +k_0d = t$. \qed
    \end{proof}

Combining the two previous results we see the set of solutions of $la+kd = \omega n$ form a lattice. Figure \ref{fig:S_7(1,3)} shows the lattice $\bar{S}_7(1,3)$.
\begin{figure}[htb]
    \centering
    \includegraphics[width=\linewidth]{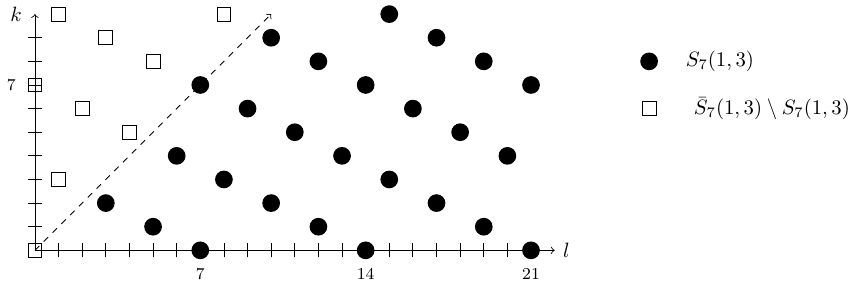}
    \caption{The lattice $\bar{S}_7(1,3)$.}
    \label{fig:S_7(1,3)}
\end{figure}

    \begin{corollary}
        The set $\bar{S}_n(a,b)$ is a lattice in $\bZ^2$ containing $(0,0)$. In particular, there are integers $l_0,k_0$ such that $l_0a+k_0d = gn$, and $\bar{S}_n(a,b) =\{x(d',-a')+y(l_0,k_0):x,y\in\bZ\}$. The matrix 
        \begin{equation}\label{eq:matrix m}
            \mathbf{M} = \begin{bmatrix}
           d' & l_0 \\
           -a' & k_0
         \end{bmatrix}^{-1}
     = \frac{1}{
     n}\begin{bmatrix}
           k_0 & -l_0 \\
           a' & d'
         \end{bmatrix}
        \end{equation}
        is a bijection  from $\bar{S}_n(a,b)$ to $\bZ^2$.
    \end{corollary}

Note that the choice of $l_0,k_0$ is not unique, in fact it need not satisfy the constraints for actual orbits. Applying $\mathbf{M}$ to $(l,k)\in\bSnab$, the $y$-coordinate represents an increment in the winding number.  Precisely, $y = (la'+kd')/n=\omega/g$ is the number of repetitions of a step sequence of length $l_0$ with $b$-count $k_0$ (followed by possible rearrangement), while the $x$-coordinate, $x = (lk_0-l_0k)/n$, represents how many times the length and $b$-count have been shifted away from $yl_0$ and $yk_0$ respectively, while preserving the winding number.

While a periodic orbit corresponds to one point $(l,k)$, there are $\binom{l}{k}$ orderings of the corresponding step sequence and the majority of these orderings correspond to distinct primitive periodic orbits with respect to a given starting vertex. 
To take this into account when counting periodic orbits we will employ a connection to Lyndon words.

\section{Lyndon Words of Fixed Letter Count}\label{sec: Lyndon}

We can view the step sequences of circuits on $C_n(a,b)$ as words on the alphabet $\{a,b\}$, with lexicographic order $a \prec b$. In this section we evaluate the number of Lyndon words with a fixed $b$-count, which will be applied in the formula for $|\mPlk|$ in the next section. 

As with cycles, for a word $w = w_1w_2\dots w_n$ we use $[w]$ for the equivalence class of $w$ under rotation, that is, $[w] = \{\sigma^s(w):s = 0,\dots,l-1\}$, where $l$ is the length of $w$ and $\sigma(w_1w_2\dots w_l) = w_2w_3\dots w_lw_1$. Note that $w$ is primitive if and only if $|[w]| = l$.
A word $w$ is a \textit{Lyndon word} if it precedes all of its rotations in lexicographic order, i.e., $w\prec \sigma^s(w)$ for $s = 1,\dots,l-1$ \cite{lyndon1954burnside}. Note that all Lyndon words are primitive. Given a word $w$, if $[w]$ contains a Lyndon word, we call that word the \textit{Lyndon rotation} of $w$, denoted $\lyn(w)$.

We will write $\allwords$ for the set of words on $\{a,b\}$ of length $l$ with $b$-count $k$, and write $\primwords$ and $\nonpwords$, and $\lyndonwords$ for the subsets of primitive, nonprimitive, and Lyndon words, contained in $\allwords$ respectively.

   \begin{lemma}\label{LyndonBasic}The order of the set of Lyndon words with fixed $b$-count is,
   \begin{equation}
       |\lyndonwords| = \frac{1}{l}\left[\binom{l}{k} - |\nonpwords|\right] \ .
   \end{equation}
    \end{lemma}

    \begin{proof}
        There is a natural bijection between $\lyndonwords$ and the set of equivalence classes $[w]$ contained in $\primwords$, namely $w\mapsto[w]$ with inverse $[x] \mapsto \lyn(x)$. Since $w$ is primitive, note $|[w]| = l$, and since $|\allwords| = \binom{l}{k}$, we have 
        \begin{equation}
            |\lyndonwords| = \frac{1}{l}|\primwords| = \frac{1}{l}\left[\binom{l}{k}-|\nonpwords|\right] \ .
        \end{equation} \qed
    \end{proof}

We can count nonprimitive words by breaking $\nonpwords$ into subsets depending on the repetition number. 
%
Denote $\ptwo(l,k,p) = \{w\in\nonpwords: w = x^p\}$. Note that $x$ need not be primitive, i.e., $p$ need not be the repetition number of $w$.  For $p>1$, we have $\ptwo(l,k,p) = \{x^p: x\in\mathbb{W}_2(l/p,k/p)\}$ and thus, 
\begin{equation}
    |\ptwo(l,k,p)| = \binom{l/p}{k/p} \ .
\end{equation}
Clearly $\ptwo(l,k,p)$ is the empty set unless $l$ and $k$ are divisible by $p$, which implies,
\begin{equation}
    \nonpwords = \ptwo(l,k,1) = \bigcup_{p>1, \, p|\gcd(l,k)}\ptwo(l,k,p)\ .
\end{equation}

Notice that this union is not disjoint. For example, $(abb)^{10}\in\ptwo(30,20,10)$, which can also be written as $(abbabb)^{5}\in\ptwo(30,20,5)$ or $(abbabbabbabbabb)^{2}\in\ptwo(30,20,2)$. 

\begin{proposition}\label{prime-nonprim}
    Suppose $p_1$ and $p_2$ are distinct divisors of $l$ and $k$ greater than $1$. Then $p_2|p_1$ if and only if $\ptwo(l,k,p_1)\subseteq \ptwo(l,k,p_2)$. 
\end{proposition}

\begin{proof}
    First, we assume $0<k<l$, as $k = 0$ or $k = l$ correspond to the trivial cases $\allwords = \{a^l\}$ or $\{b^l\}$.
    We will denote $l_1 = l/p_1$ and similarly $l_2= l/p_2$, with $k_1 = k/p_1,$ and $k_2=k/p_2$.  The conditions $p_2|p_1$, $l_1|l_2$, and $k_1|k_2$ are all equivalent.
    Now suppose $p_2|p_1$, so $p_1 = zp_2$ for some $z \in \bZ$. For all $w\in\ptwo(l,k,p_1)$, there exists a word $x$ such that $w = x^{p_1} = (x^{z})^{p_2}\in\ptwo(l,k,p_2)$.

    Conversely, suppose that $p_2\nmid p_1$. Let $x = a^{l_1-k_1}b^{k_1}$ and $w = x^{p_1}$. Suppose that $w\in\ptwo(l,k,p_2)$. Then $w = x^{p_1} = y^{p_2}$, where the length and $b$-count of $y$ are $l_2$ and $k_2$, respectively. If $l_2 > l_1$, then $y$ begins with $x^z$ for some $z\geq 1$. Since $l_1\nmid l_2$, the word $y$ terminates partway through a repetition of $x$. Since the last letter of $w$ is $b$, the last letter of $y$ must also be $b$, that is, $y = x^za^{l_1-k_1}b^s$ where $0 < s < k_1$. Since the first letter of $x$ (and thus of $y$) is $a$, we know $w = y^{p_2}$ contains the string $ab^sa$, which is a contradiction since $w$ is a repetition of $a^{l_1-k_1}b^{k_1}$. If $l_2 < l_1$, then $y$ is a substring of $x$. Since $y$ must contain a $b$, we have $y = a^{l_1-k_1}b^s$ for $0<s<k_1$, which again implies the contradiction that $w$ contains the string $ab^sa$.\qed
\end{proof}

\begin{corollary} Let $p_1,\dots,p_d$ be the prime divisors of $\gcd(l,k)$. Then 
\begin{equation}
\nonpwords = \bigcup_{j=1}^d\ptwo(l,k,p_j) \ .
\end{equation}
\end{corollary}

To count $\nonpwords$ it remains to quantify the intersection of $\ptwo(l,k,p)$ for $p$ a prime divisor of $\gcd(l,k)$. In the previous example $(abb)^{10}\in \ptwo(30,20,5)\cap \ptwo(30,20,2)$ which we will see is precisely $\ptwo(30,20,10)$.

\begin{proposition}
    Suppose $p_1$ and $p_2$ are coprime and divisors of $l$ and $k$. Then 
    \begin{equation}
        \ptwo(l,k,p_1)\cap\ptwo(l,k,p_2)=\ptwo(l,k,p_1p_2) \ .
    \end{equation}
\end{proposition}

\begin{proof}
    By proposition \ref{prime-nonprim} we have $\ptwo(l,k,p_1p_2)\subseteq\ptwo(l,k,p_1)\cap\ptwo(l,k,p_2)$.
    Suppose $w = w_1\dots w_l \in \ptwo(l,k,p_1)\cap\ptwo(l,k,p_2)$. Let $l_1 = l/p_1,$ with $l_2 = l/p_2$, and $l_3 = l/(p_1p_2)$. Note that $w = \sigma^{l_1}(w) = \sigma^{l_2}(w)$, that is, for each $j\in\{1,\dots ,l\}$ and each $z\in\bZ$, $w_j = w_{j+zl_1} = w_{j+zl_2}$.
    Since $\gcd(p_1,p_2) = 1$, there exist $u,v\in\bZ$ such that $up_1 + vp_2 = 1$. Multiplying by $l_3$ on both sides, $l_3 = ul_2+vl_1$. Thus, for each $j$, we have $w_{j+l_3} = w_{j+ul_2+vl_1} = w_j$. In other words, $w = \sigma^{l_3}(w)$, so $w \in \ptwo(l,k,p_1p_2)$. \qed
\end{proof}

Using the previous results,
\begin{align}
    |\nonpwords| &= \left|\bigcup_{j = 1}^d\ptwo(l,k,p_j)\right| \\ 
    &=
            \sum_{m|\gcd(l,k),m>1}-\mu(m)|\ptwo(l,k,m)| \\
            &=
            \sum_{m|\gcd(l,k),m>1}-\mu(m)\binom{l/m}{k/m} \ ,
\end{align}
where $\mu$ is the M\"obius function,
\begin{equation}
    \mu(m) = \begin{cases}
        1 & \text{if } m = 1\\
        (-1)^h & \text{if $m$ is the product of $h$ distinct primes}\\
        0 & \text{if $m$ is divisible by a square $>1$} 
    \end{cases} \ .
\end{equation}
Theorem \ref{thm: Lyndon} follows immediately.

To illustrate this result, consider $l = 9$ and $k = 3$, with $|\mathbb{W}_2(9,3)| = 84$. In this case $m = 1,3$ and $\mathbb{P}^c_2(9,3,3) = \{(aab)^3,(aba)^3,(baa)^3\}$. The other $81$ words are primitive, and up to rotation there are $9$ distinct groups of $9$ words, with one word in each group being Lyndon. In particular, 
\begin{equation}
|\bL_2(9,3)| = \frac{1}{9}\left[\binom{9}{3} -\binom{3}{1}\right] = 9 \ .
\end{equation}
Figure \ref{fig:914examples} shows the primitive periodic orbits on $C_9(1,4)$ corresponding to these $9$ words. 
In the next section we define a natural bijection between $\lyndonwords$ and the set of primitive periodic orbits on $C_n(a,b)$ starting at a given vertex.

\begin{figure}[htb]
    \captionsetup[subfigure]{justification=centering}
    \centering
    \begin{subfigure}{0.2\textwidth}
        \includegraphics[width=\linewidth]{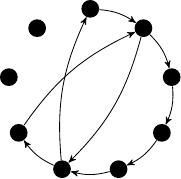}
        \caption*{$\varphi(111111444,0)$}
    \end{subfigure} \hspace{0.1\textwidth}
    \begin{subfigure}{0.2\textwidth}
        \includegraphics[width=\linewidth]{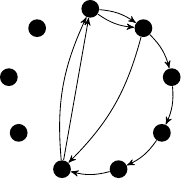}
        \caption*{$\varphi(111114144,0)$}
    \end{subfigure} \hspace{0.1\textwidth}
    \begin{subfigure}{0.2\textwidth}
        \includegraphics[width=\linewidth]{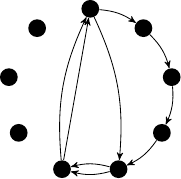}
        \caption*{$\varphi(111114414,0)$}
    \end{subfigure}
    
    \vspace{2em}
    \begin{subfigure}{0.2\textwidth}
        \includegraphics[width=\linewidth]{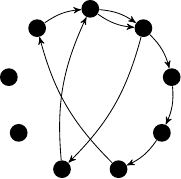}
        \caption*{$\varphi(111141144,0)$}
    \end{subfigure} \hspace{0.1\textwidth}
    \begin{subfigure}{0.2\textwidth}
        \includegraphics[width=\linewidth]{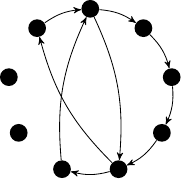}
        \caption*{$\varphi(111141414,0)$}
    \end{subfigure} \hspace{0.1\textwidth}
    \begin{subfigure}{0.2\textwidth}
        \includegraphics[width=\linewidth]{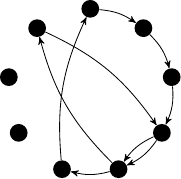}
        \caption*{$\varphi(111144114,0)$}
    \end{subfigure}
    
    \vspace{2em}
    \begin{subfigure}{0.2\textwidth}
        \includegraphics[width=\linewidth]{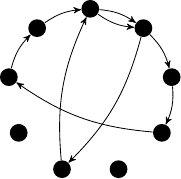}
        \caption*{$\varphi(111411144,0)$}
    \end{subfigure} \hspace{0.1\textwidth}
    \begin{subfigure}{0.2\textwidth}
        \includegraphics[width=\linewidth]{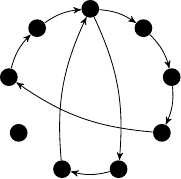}
        \caption*{$\varphi(111411414,0)$}
    \end{subfigure} \hspace{0.1\textwidth}
    \begin{subfigure}{0.2\textwidth}
        \includegraphics[width=\linewidth]{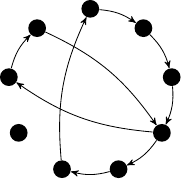}
        \caption*{$\varphi(111414114,0)$}
    \end{subfigure}

\caption{The periodic orbits $\varphi(w,0)$ for each $w\in\bL_2(9,3)$, where $G = C_9(1,4)$. Repeated bonds are shown twice. The bonds are the same in $\varphi(111111444,0)$ and $\varphi(111141414,8)$, but these orbits are distinct due to the orders of step sequences.}
\label{fig:914examples}
\end{figure}
\section{Counting Periodic Orbits}\label{sec: counting}

Let $v$ be a vertex in $C_n(a,b)$. For $w\in\allwords$, let $\psi(w,v)$ denote the unique path $c$ starting at $v$ such that $\st(c) = w$. Note that $\psi(w,v)$ is a circuit if and only if $(l,k)\in\Snab$, in which case we denote $\varphi(w,v) = [\psi(w,v)]$, see
figure \ref{fig:914examples}. Clearly $\varphi$ is a surjective function, from $\allwords \times \bZ_n$ to the set of all periodic orbits of length $l$ and $b$-count $k$, as each periodic orbit contains a circuit with some initial vertex and step sequence. If $c = \psi(w,v)$ is nonprimitive, then the step sequence $w$ is also nonprimitive. The converse, however, is not true, since there are $n$ distinct bonds corresponding to each step size. For example, in $C_9(1,4)$, the step sequence $w=(114)^3$ is primitive, but $\psi(w,v)$ is a nonprimitive circuit for any $v$ (see figure \ref{fig:114114114}). Moreover, $\varphi:\primwords\times\bZ_n \to \mPlk$ is not injective, since for any orbit $p = [c]$, we also have $p = \varphi(\sigma^j(\st(c)),v_j)$ where $v_j$ is the $j$-th vertex in $c$. Replacing primitive words with Lyndon words corrects this by keeping only one rotation of a primitive step sequence.

\begin{figure}[htb]
    \centering
    \includegraphics[width=0.3\linewidth]{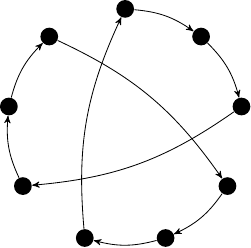}
    \caption{The periodic orbit $\varphi(114114114,0)$ on $C_9(1,4)$, which is a primitive periodic orbit with a nonprimitive step-sequence.}
    \label{fig:114114114}
\end{figure}

\begin{lemma} The map
        $\varphi: \mathbb{L}_2(l,k)\times \bZ_n \to \mPlk$ is injective.
    \end{lemma}

    \begin{proof}
        Let $w,y\in \mathbb{L}_2(l,k)$ and $v,u\in\mathbb{Z}_n$. Write $c = \psi(w,v)$ and $d = \psi(y,u)$, so $[c] = \varphi(w,v),$ and $[d] = \varphi(y,u)$. If $[c] = [d]$, there exists $s \in \{0,\dots ,l-1\}$ such that $c = \sigma^s(d)$. Then $y = \sigma^s(w)$, and as $y$ and $w$ are Lyndon words, $y = w$, consequently $s = 0$, and $u = v$. \qed
    \end{proof}

Since $\varphi:\lyndonwords\times\bZ_n \to \mPlk$ is not surjective, we add back nonprimitive words which map to primitive circuits. For example, each orbit in figure \ref{fig:914examples} has $9$ distinct rotations, one for each starting vertex, while the orbit in figure \ref{fig:114114114} only has $3$ distinct rotations, since $3$ is the repetition number of its step sequence. We classify nonprimitive words by their primitive roots, with a condition on the winding number which identifies those words which map to nonprimitive orbits.

\begin{proposition}\label{winding vs repetition}
    Suppose $(l,k)\in\Snab$, that is, $la+kd = \omega n$ for some $\omega$, and $w = x^r\in \allwords$ where $r$ is the repetition number of $w$. Then, for a given vertex $v$, the periodic orbit $p = \varphi(w,v)$ is primitive if and only if $\gcd(r,\omega) = 1$. In general, if $p = q^t$ for some periodic orbit $q$, then $q$ is primitive if and only if $t = \gcd(r,\omega)$. 
\end{proposition}
\begin{proof}
    Suppose $\gcd(r,\omega) = h > 1$. Let $s = r/h$, and $c = \psi(x^s,v)$. Note that $c$ is a path of length $l/h$ and $b$-count $k/h$, so the transit distance $  \Delta(c) = \Delta(p)/h = \omega n/h$. Since $h|\omega$, the path $c$ is a circuit. Then $p=[c^h]$ is not primitive. 

    Conversely, suppose $p$ is not primitive. Then there is a primitive circuit $c$ such that $p = [c^s]$ and $s>1$. In particular, $(l/s,k/s)\in S_n(a,b)$ with winding number $\omega/s$, so $s|\omega$. Now it suffices to show $s|r$. Write $y = \st(c)$ and $l_x,l_y$ for the length of $x$ and $y$, respectively. Since $l = l_ys = l_xr$, it suffices to show that $l_x|l_y$. As $x$ is the root of $w$, we have $l_x \leq l_y$. For the sake of contradiction, suppose $l_x \nmid l_y$. Then $l_y = \mu l_x+ \rho$ for some $\mu > 0$, $0<\rho<l_x$. As in the proof of proposition \ref{prime-nonprim}, $y$ is a repetition of $x$ followed by a truncation of $x$, so $y = x^\mu x_1\dots x_\rho$. Thus, as $y^s=x^r$, we have,
    \begin{align}
        x^\mu x_1\dots x_\rho x_1x_2x_3\dots  &= x^\mu x_1\dots x_\rho x_{\rho+1}x_{\rho+2}x_{\rho+3}\dots \ .
    \end{align}
    Hence, $x = \sigma^\rho(x)$ for $0<\rho < l_x$, which is a contradiction since $x$ is primitive.

    Now suppose $p = q^t$ where $q$ is a periodic orbit. Since $x$ is primitive, we have $q = \varphi(x^{r/t},v)$ where $r/t$ is the repetition number of $\st(q)$. Note that $\omega/t$ is the winding number of $q$, so $(l/t)a + (k/t)d = (\omega/t) n$. In particular, $t$ is a common divisor of $r$ and $l$. By the first part of the proposition, $q$ is primitive if and only if $\gcd(r/t,\omega /t) = 1$, or equivalently, $\gcd(r,\omega) = t$.    \qed
\end{proof}

To generate orbits that do not correspond with Lyndon words, we can use the nonprimitive words $w = x^r$ where $x$ is a Lyndon word and $r$ is coprime with $\omega$. Denote by $\bL_2^q(l,k)$ the set $\{x^q:x\in\mathbb{L}_2(l/q,k/q)\}$. Note that the notation is intentionally different to $\ptwo(l,k,p)$, since $q$, unlike $p$, must be the repetition number. In particular, note that $\bL_2^{q_1}(l,k)$ and $\bL_2^{q_2}(l,k)$ are disjoint if $q_1\neq q_2$.

Now we show that it is sufficient to count over all common divisors of $l$ and $k$ coprime with $\omega$.

\begin{lemma}\label{lem: Q}
    Suppose $la+kd = \omega n$ for some $\omega \in\bZ^+$. For each $p$ in the set $\mPlk$, the primitive periodic orbit $p = \varphi(w,v)$ for some $w\in \mathbb{L}_2^q(l,k)$ where $q|\gcd(l,k)$ and $ \gcd(q,\omega) = 1$.
\end{lemma}

\begin{proof}
    Let $p = [c]\in \mPlk$. Write $\st(c) = w_0 = x_0^q$ where $x_0$ is primitive. Then $q|\gcd(l,k)$. Since $p$ is primitive, $\gcd(q,\omega) = 1$ by proposition \ref{winding vs repetition}. Then for $w = \lyn(x_0)^q\in\bL_2^q(l,k)$, we have $p = \varphi(w,v)$ for some vertex $v$ in $c$. \qed
\end{proof}

\begin{corollary}\label{Surjective Overkill} Suppose $la+kd = \omega n$. Let $Q$ be the set of divisors of $\gcd(l,k)$ that are coprime with $\omega$. Then
\begin{equation}
    \displaystyle\varphi: \bigcup_{q\in Q}\mathbb{L}_2^{q}(l,k)\times \bZ_n \to \mPlk
    \end{equation}
    is surjective.
\end{corollary}

Conversely, each word in $\bL_2^q(l,k)$ maps  to $n/q$ distinct orbits, so restricting the starting vertices can make $\varphi$ injective.

\begin{proposition}\label{bijective phi}
    Suppose $la+kd = \omega n$. Let $Q$ be the set of divisors of $\gcd(l,k)$ that are coprime with $\omega$. Then 
\begin{equation}
    \varphi: \bigcup_{q\in Q}\mathbb{L}_2^{q}(l,k)\times \bZ_{n/q} \to \mPlk
\end{equation}
    is bijective.
\end{proposition}

\begin{proof}
    Let $w \in \bL_2^q(l,k)$ where $q\in Q$. It suffices to show that $\varphi(w,0) = \varphi(w,v)$ if and only if $v \equiv jn/q$ (\text{mod} $n$) for some $j = 0,\dots ,q-1$, as then for every orbit $p$ with step sequence $w$, there is a unique vertex $u \in\bZ_{n/q}$ such that $p = \varphi(w,u)$.

    Suppose $v \equiv jn/q$ (\text{mod} $n$) for some $j = 0,\dots ,q-1$. Since $\gcd(\omega, q) = 1$, there exist integers $m,z$ such that $zq - m\omega = j$. Note that the step sequence is a repetition $w = x^q$ where $\Delta(x) = \omega n/q$, so we have $v + m\Delta (x) = zn$ and $\varphi(w,v) = \varphi (w,v+m\Delta(x)) = \varphi (w,0)$.

    Conversely, suppose $\varphi(w,0) = \varphi(w,v)$. Then $\psi(w,0) = \sigma^s(\psi(w,v))$ for some $s = 0,\dots ,l-1$; that is, the $p$-th vertex of $\sigma^s(\psi(w,v))$ is the $(s+p)$-th vertex of $\psi(w,0)$. Since $w = x^q$ where $x$ is Lyndon, we know that $s = ml/q$ for some $m \in \{0,\dots , q-1\}$.  Thus,
    \begin{align}
        w_1 + w_2 + \dots  + w_{m(l/q)+p} &= v + w_1 + w_2 + \dots + w_{p}\\
        m(x_1+\dots+x_{l/q}) + w_1 + w_2 + \dots + w_p &= v + w_1 + w_2 + \dots + w_{p}\\
        m\Delta(x) &= v\\
        m\frac{\omega n}{q} &= v \ ,
    \end{align}
    and since $\gcd(\omega,q) = 1$, we have $v \equiv jn/q$ (\text{mod} $n$) for some $j = 0,\dots,q-1$. \qed
\end{proof}

\begin{theorem}\label{unreduced sum}
  Let $C_n(a,b)$ be a connected $2$-regular circulant digraph and 
    suppose $la+kd = \omega n$. Let $Q$ be the set of divisors of $\gcd(l,k)$ that are coprime with $\omega$. Then,  
    \begin{equation}
        |\mPlk|=\frac{n}{l}\sum_{q \in Q}\sum_{m|\gcd(l,k)/q}\mu(m)\binom{l/(qm)}{k/(qm)}\ ,    
        \end{equation}
    where $\mu$ is the M\"obius function.
\end{theorem}

\begin{proof}
    Note that the union in proposition \ref{bijective phi} is disjoint. Thus
    \begin{align}
    |\mPlk| &= \sum_{q\in Q}|\bZ_{n/q}|\cdot|\mathbb{L}_2^{q}(l,k)|\\
    &= \sum_{q \in Q}\frac{n}{q}\left|\mathbb{L}_2\left(\frac{l}{q},\frac{k}{q}\right)\right|\\
    &= n\sum_{q \in Q}\frac{1}{q} \left(\frac{1}{l/q}\sum_{m|\gcd(l,k)/q}\mu(m)\binom{l/(qm)}{k/(qm)}\right) \ .
\end{align} \qed
\end{proof}

\begin{proposition}\label{sum reduction thm}
    Let $\gamma,\omega\in\bN$ and $Q$ be the set of 
     divisors of $\gamma$ that are coprime with $\omega$.  
    For any function $f:\bN \to \mathbb{R}$, 
    \begin{equation}\label{eq:sum reduction}
        \sum_{q\in Q}\sum_{s|\gamma/q}\mu(s)f(qs) = \sum_{m|\gcd(\gamma,\omega)}\mu(m)f(m) \ . 
    \end{equation}
\end{proposition}

\begin{proof}
Fix $x=q_0 s_0$ for some $q_0\in Q$ and $s_0|\gamma/q_0$, so $x|\gamma$.  Let $y$ be the  greatest divisor of $x$ coprime with $\omega$.  Then $x=yz$ for some $z\in \mathbb{N}$.  If $qs=x$ for $q\in Q$ and $s|\gamma/q$ then $s=zj$ and $j|y$.
On the other hand, if $j|y$ we can set $s=zj$ and $x$ can be factored so $x=s(y/j)$ with $y/j\in Q$.  Hence, the coefficient of $f(x)$ in the left hand side of (\ref{eq:sum reduction}) is
\begin{align}\label{eq:mobius property}
    \sum_{j|y} \mu (zj) &= \mu(z)\sum_{j|y} \mu(j) 
\end{align}
as the M\"obius function is multiplicative. As $\sum_{j|y} \mu(j) = 0$ if $y\neq 1$, the coefficient of $f(x)$ is zero unless $x|\omega$, in which case the coefficient is $\mu(x)$.
\qed
\end{proof}

Combining theorem \ref{unreduced sum} and proposition \ref{sum reduction thm}, with $\gamma = \gcd(l,k)$ and $f(x) = \binom{l/x}{k/x}$ we obtain theorem \ref{thm: old main}.    A periodic orbit $p$ of length $l$ and winding number $\omega$ has $b$-count $k_\omega = (\omega n - la)/d \leq l$, see equation (\ref{eq: la+kd}). In particular, such an orbit exists if and only if $k_\omega$ is an integer and $la/n \leq \omega \leq lb/n$. Using theorem $\ref{thm: old main}$ and summing over the allowed values of the $b$-count 
we have,
    \begin{equation}
        |\mathcal{P}_l(C_n(a,b))| = \frac{n}{l}\sum_{\omega = \lceil la/n \rceil}^{\lfloor lb/n \rfloor} \left( \sum_{m|\gcd(l,k_\omega,\omega)}\mu(m)\binom{l/m}{k_\omega/m}\right) \ . 
    \end{equation}
    Finally, note that the divisors $m$ of $\omega$ for which $\binom{l/m}{k_\omega/m}\neq 0$ are precisely the divisors of $\gcd(l,k_\omega,\omega)$ which proves theorem \ref{thm: new main}.
    
\section{Concluding Remarks}\label{sec: conclusions}

For $2$-regular directed circulant graphs of arbitrary order and step sizes, we have evaluated the number of primitive periodic orbits of a given length. In doing so, we have described the path lengths and step size counts for which cycles exist, and obtained a formula for the number of Lyndon words on two letters with a given length and letter count.

 To illustrate these results consider periodic orbits on 
 the circulant graph $C_{440}(5,14)$ of length $l = 360$ and $b$-count $k = 240$. 
 The winding number for these orbits is $\omega = 9$. From lemma \ref{lem: Q} any step sequence $w$ in $\mathbb{W}_2(360,240)$ corresponding to a primitive orbit will have a repetition number $q$ which divides $l$ and $k$, and is coprime with $\omega$, i.e., $q \in Q = \{1,2,4,8,5,10,20,40\}$. For each $q$, we can evaluate $|\mathbb{L}_2(360/q,240/q)|$ the number of words where the root is a Lyndon word of length $360/q$ and $b$-count $240/q$, using theorem \ref{thm: Lyndon}.  To simplify the formulas let $B_t = \binom{360/t}{240/t}$.
\begin{align}
    q& = 1& 
    |\mathbb{L}_2(360,240)| &= \frac{1}{360}\left(B_1 - B_2 - B_3 - B_5 + B_6 + B_{10} + B_{15} - B_{30}\right) \\
    q &= 2& 
    |\mathbb{L}_2(180,120)| &= \frac{1}{180}\left(B_2 - B_4 - B_6 - B_{10} + B_{12} + B_{20} \right. \nonumber \\
    &&& \qquad \qquad \qquad \qquad \qquad \qquad \qquad \quad \left. + B_{30} - B_{60}\right)\\
    q &= 4 &
    |\mathbb{L}_2(90,60)| &= \frac{1}{90}\left(B_4 - B_8 - B_{12} - B_{20}  + B_{24} + B_{40} \right. \nonumber \\
    &&& \qquad \qquad \qquad \qquad \qquad \qquad \qquad \quad \left. + B_{60} - B_{120}\right)\\
    q &= 8& 
    |\mathbb{L}_2(45,30)| &= \frac{1}{45}\left(B_8 - B_{24} - B_{40} + B_{120}\right)\\
    q &= 5& 
    |\mathbb{L}_2(72,48)| &= \frac{1}{72}\left(B_5 - B_{10} - B_{15} + B_{30}\right)\\
    q &= 10& 
    |\mathbb{L}_2(36,24)| &= \frac{1}{36}\left(B_{10} - B_{20} - B_{30} + B_{60}\right)\\
    q &= 20& 
    |\mathbb{L}_2(18,12)| &= \frac{1}{18}\left(B_{20}-B_{40}-B_{60}+B_{120}\right)\\
    q &= 40& 
    |\mathbb{L}_2(9,6)| &= \frac{1}{9}\left(B_{40}-B_{120}\right)
\end{align}

To count primitive orbits, each of these quantities is multiplied by $n/q$ (the number of inequivalent starting vertices) and summed, as in the proof of theorem \ref{unreduced sum}.  After cancellation, using proposition \ref{sum reduction thm}, we see that the number of primitive orbits is,
\begin{equation}
    |\mathcal{P}_{360,240}(C_{440}(5,14))| = \frac{440}{360}\left(B_1 - B_3\right) = \frac{440}{360}\left(\binom{360}{240}-\binom{120}{80}\right).
\end{equation}

The cancellation in proposition \ref{sum reduction thm}, is illustrated in figure \ref{fig:binom grid}.
The figure shows all values of $t = qm$ with $q\in Q$ and $m$ a square-free divisor of $\gcd(l/q,k/q)$.  The arrows in the figure join pairs of values that differ by multiplication by a prime. For each $t$ let $H_t$ denote the corners of the hypercube where $t$ is the largest corner. For example, $H_{60} = \{2,4,6,10,12,20,30,60\}$. The intersection of $H_t$ with $Q$ is the set of values of $q$ which contribute $\pm B_t$ to the sum in proposition \ref{sum reduction thm}. 
The sign $\mu(m)$ depends on the parity of the number of prime factors of $m$, that is, the graph distance from $t$ to $q$. Counting points on a cube with this alternating sign is equivalent to an alternating sum of binomial coefficients, which is nonzero if and only if $|H_t \cap Q| = 1$. This corresponds to $\sum_{j|y}\mu(j) = 0$ if $y \neq 1$ used in proving proposition \ref{sum reduction thm}. In the given example this happens twice: $H_3\cap Q=\{1\}$ and $H_1\cap Q=\{1\}$.  This condition is equivalent to $t$ being coprime with $Q$. As $m$ is square-free, the set of contributors $\{1,3\}$ is the set of square-free common divisors of $\omega,l,k$, as in theorem \ref{thm: old main}.  

\begin{figure}[htb]
    \centering
    \includegraphics[width=0.8\linewidth]{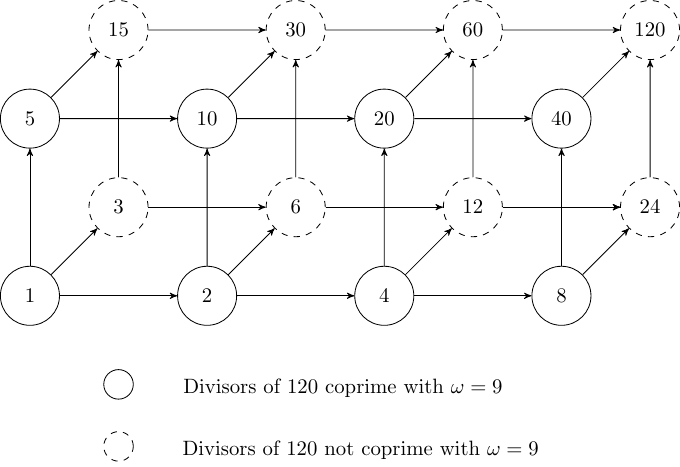}
    \caption{Divisors $t = qm$ of $\gcd(360,240)=120$ with $q$ coprime with the repetition number $\omega=9$ and $m$ a square-free divisor of $\gcd(360/q,240/q)$.  Arrows show two divisors are related by multiplication by a prime.}
    \label{fig:binom grid}
\end{figure}

To demonstrate counting all primitive periodic orbits of a fixed length consider $C_{21}(4,10)$ with $l=15$. The lattice of allowed winding numbers and $b$-counts $S_{21}(4,10)$ is shown in figure \ref{fig:S_21(4,10)}. The figure also illustrates the bounds on the winding number $la/n \leq \omega \leq lb/n$.   
For instance, counting orbits of length $l = 15$, we have $2.86 \leq \omega \leq 7.14$. 
The only valid winding numbers are those divisible by $\gcd(4,10) = 2$, which is reflected in the proof of theorem \ref{thm: new main}; since $k_3 = 1/2$ and $k_5 = 15/2$ and $k_7 = 29/2$, which are not integers, the corresponding binomial coefficients are omitted. Similarly, when $\omega = 4$ or $6$, note that $m = 2$ has $l/m = 15/2$ and $m = 3$ has $k_6/m = 11/3$.  Therefore,
\begin{align}
    |\mathcal{P}_{15}(C_{21}(4,10))| &= \frac{21}{15}\sum_{\omega = 3}^7 \sum_{m|\omega}\mu(m) \binom{15/m}{(21\omega - 60)/(6m)}\\
    &= \frac{21}{15}\left(\binom{15}{4} + \binom{15}{11}\right)\\
    &= 3822 \ .
\end{align}

\begin{figure}[htb]
    \centering
    \includegraphics[width=0.5\linewidth]{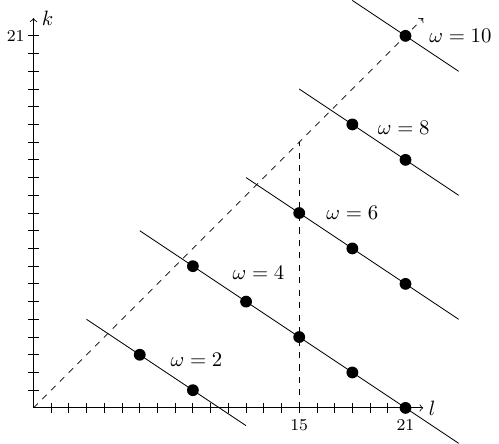}
    \caption{The lattice $S_{21}(4,10)$.}
    \label{fig:S_21(4,10)}
\end{figure}

%


\begin{acknowledgements}
TH would like to thank Lauren Engelthaler and Isaac Hellerman for helpful suggestions.
\end{acknowledgements}

%
%

\bibliographystyle{spmpsci}      
\bibliography{refs}   

%
%

\end{document}